\documentclass[secthm,seceqn,amsthm,ussrhead]{amsart}
\usepackage{amsmath,latexsym}
\usepackage[psamsfonts]{amssymb}
\usepackage [dvips]{graphicx,color}
\usepackage{times}
\usepackage[mathcal]{euscript}
\numberwithin{equation}{section} \textwidth=140mm \textheight=200mm
\renewcommand{\epsilon}{\varepsilon}
\newcommand{\be}{\begin{equation}}
\newcommand{\ee}{\end{equation}}

\newcommand{\R}{\mathbb{R}}

\newcommand{\T}{\mathbb{T}}

\newcommand{\Z}{\mathbb{Z}}
{\bf}{\it}
\newtheorem{theorem}{Theorem}[section]
\newtheorem{lemma}[theorem]{Lemma}
\newtheorem{corollary}[theorem]{Corollary}

\newtheorem{remark}[theorem]{Remark}

\date{\today}
\begin{document}
\title{The existence and location of eigenvalues of the
        one particle discrete Schr\"{o}dinger operators}

\author{Saidakhmat N.~Lakaev,~ Ender Ozdemir}
     \date{\today}
\begin{abstract}
We consider a quantum  particle moving in the one dimensional
lattice $\Z$ and interacting with a indefinite sign external field
$\hat v$. We prove that the associated discrete Schr\"{o}dinger
operator $H$ can have one or two eigenvalues, situated as below the
bottom of the essential spectrum, as  well as above its top.
Moreover, we show that the operator $H$ can have two eigenvalues
outside of the essential spectrum such that one of them is situated
below the bottom of the essential spectrum, and other one  above its
top.
\end{abstract}
\maketitle
\section{Introduction}
We consider the discrete Schr\"{o}dinger operators $H$ of a quantum
particle moving in the one-dimensional lattice $\Z$ and interacting
with a {\it indefinite sign} external field $V(x),x\in\Z$, i.e., the
potential has  positive and negative values.

In \cite{BSimon}  of B.Simon the existence of eigenvalues of a
family of continuous Schr\"{o}dinger ope\-ra\-tors
$H=-\Delta+\lambda V,\,\lambda>0$ in one and two-dimensional cases
have been considered. The result that $H$ has bound state for all
$\lambda>0$ if only if  $\int V(x)dx <0$ is proven there for all
$V(x)$ with  $\int (1+|x|^2)|V(x)|dx <+\infty.$

In \cite{MKlaus} it is presented that under certain conditions on
the potential a one-dimensional Schr\"{o}dinger operator has
a unique bound state in the limit of weak coupling while under
other conditions no bound state in this limit. This question is
studied for potentials obeying $\int (1+|x|)|V(x)|dx <+\infty.$

The questions further discussed by R.Blankenbecker, M.N.Goldberger
and B.Simon \cite{BGSimon}.

All these results require the use of the modified  determinant.

Throughout physics, stable composite objects are usually formed by
the way of attractive forces, which allow the constituents to lower
their energy by binding together. Repulsive forces separate
particles in free space. However, in structured environment such as
a periodic potential and in the absence of dissipation, stable
composite objects can exist even for repulsive interactions
\cite{WTL}.

The Bose-Hubbard model, i.e., the Schr\"{o}dinger operators on
lattice, which have been used to describe the repulsive pairs, is a
theoretical basis for explanation of the experimental results
obtained in \cite{WTL}.

Since the continuous Schr\"{o}dinger operator has essential
spectrum fulfilling semi-axis $[0,+\infty)$ and its eigenvalues
appear below the bottom of the essential spectrum, it is a model,
 which well described the systems of two-particles with the attractive interactions.

Zero-range potentials are the mathematically correct tools for
describing contact interactions. The latter reflects the fact that
the zero-range potential is effective only in the s-wave
\cite{{Yakovlev}}.

The existence of eigenvalues of a family of Schr\"{o}dinger
ope\-ra\-tors $H=-\Delta-\mu V,\,\lambda>0$ with a of rank one
perturbation $V$ in one and two-dimensional lattices have been
considered in \cite{LKhL2012}. The result that $H$ has a unique
bound state for all $\mu>0$ is proven and an asymptotics for the
unique eigenvalue $e(\mu)$ lying below the bottom of the essential
spectrum is found as $\mu\to 0$.

In \cite{KhL2013} for the Hamiltonian  $H$ of two fermions with
attractive interaction on a neighboring sites in the one-dimensional
lattice $\Z$ has been considered and  an asymptotics of the unique
eigenvalue lying below the bottom  of its essential spectrum has
been proven.

For a family of the generalized Friedrichs models
$H_{\mu}(p)$, $ \mu>0, p\in\T^2$ with the perturbation of rank one, associated to a
system of two particles moving on the two-dimensional lattice $\Z$
has been considered in \cite{LIK2012} and the existence or absence of a
positive coupling constant threshold $\mu=\mu_0(p)>0$ depending on
the parameters of the model has been proved.

The main goal of the report is to state the existence and location
of eigenvalues of the discrete Schr\"odinger operator
$H=\hat{H}_{\mu\lambda}:=\hat{H}_0+\hat{V}_{\mu\lambda},$ with the
zero-range interaction $\mu\neq0$ and with interactions
$\lambda\neq0$ on a neighboring sites. We establish that the
operator $H$ may have one or two eigenvalues, situating as below the
bottom of the essential spectrum, as  well as above its top.
Moreover, the operator $H$ can have two eigenvalues outside of the
essential spectrum, where one of them is situated below the bottom
of the essential spectrum and other one above its top.

The results are contradict and in this case improved the known
results of \cite{BSimon,MKlaus,BGSimon} for continuous Schr\"odinger
operators.

\section{The coordinate representation of the one particle discrete Schr\"odinger operator}
Let $\mathbb{Z}$ be the one dimensional lattice(integer numbers )
and $\ell^{2}(\mathbb{Z}) $ be the Hilbert space of square summable
functions on $\mathbb{Z}$ and $\ell^{2,e}(\mathbb{Z}
)\subset\ell^{2}(\mathbb{Z} )$ be the subspace  of
functions(elements) $\hat f\in \ell^{2}(\mathbb{Z})$ satisfying the
condition
\begin{equation*}
\hat f(x)=\hat f(-x),\,\,x\in\Z.
\end{equation*}
The one particle discrete Schr\"odinger operator
$\hat{H}_{\mu\lambda}$ acting in $\ell^{2,e}(\mathbb{Z})$ has the
form
\begin{equation}\label{hamiltonianh}
\hat{H}_{\mu\lambda}:=\hat{H}_0+\hat{V}_{\mu\lambda},
\end{equation}
where $\hat{H}_{0}$ is the Teoplitz type operator
\begin{equation}\label{hamiltonianh_0}
(\hat{H}_0 \hat{\varphi})(x): = \sum_{s\in{Z}
}\hat{\varepsilon}(s)\hat\varphi(x+s),\,\,\hat{\varphi}\in
\ell^{2,e}(\mathbb{Z} ),
\end{equation}
and
\begin{equation}\label{potential}
(\hat{V}_{\mu\lambda}\hat{\varphi})(x):=\hat{v}_{\mu\lambda}
(x)\hat{\varphi}(x),\quad \hat{\varphi}\in
\ell^{2,s}(\mathbb{Z} ).
\end{equation}

The functions $\hat{\varepsilon}(s)$ and $\hat{v}_{\mu\lambda}(s)$ are
defined on $\mathbb{Z}$ as follows
$$
\hat{\varepsilon}(s)=
\left\lbrace\begin{array}{ccc}
1,\,\,\quad  |s|=0\\
-\frac{1}{2},\,\, \quad  |s|=1\\
0,\quad   |s|>1,
\end{array}\right.$$
and
$$ \hat{v}_{\mu\lambda}(s)=
\left\lbrace\begin{array}{ccc}
\mu,\quad |s|=0\\
\frac {\lambda}{2},\quad |s|=1\\
0, \quad  |s|>1,
\end{array}\right.
$$
where
$\mu,\,\lambda \in\R
$ are real numbers.

We remark that $\hat{H}_{\mu\lambda}$
is a bounded self-adjoint operator on
 $\ell^{2,e}(\mathbb{Z}).$

\section{The momentum representation of the  discrete Schr\"odinger operator}\label{momentum}

Let $ \mathbb{T}=(-\pi;\pi]$ be the one dimensional torus and $
L^2(\mathbb{T},d\nu)$ be the Hilbert space of integrable functions
on $\mathbb{T}$, where $d\nu$ is the (normalized) Haar measure on
$\T,$\,\,\, $d\nu(p) =\dfrac {dp}{2\pi}$.

Let $L^{2,e}(\mathbb{T},d\nu)\subset{L^{2}(\mathbb{T},d\nu )}$
be the subspace of elements $f\in
L^2({\T},d\nu)$ satisfying the condition
\begin{equation*}
f(p)=f(-p),\quad \text{a.e.}\,\,p\in\T.
\end{equation*}
In the momentum representation the operator ${H}_{\mu\lambda}$ acts
on $L^{2,e}(\mathbb{T},d\nu)$ and is of the form
$$
H_{\mu\lambda}=H_0+V_{\mu\lambda},
$$
where $H_0$ is the multiplication operator by function $\varepsilon(p)=1-\cos p$:
$$
(H_0 f)(p)=\varepsilon(p)f(p), \quad f\in L^{2,e}(\mathbb{T},d\nu),
$$
and  $V_{\mu\lambda}$ is the integral operator of rank $2$
$$
(V_{\mu\lambda}f)(p)= \int\limits_{\mathbb{T}}
\Big(\mu+\lambda \cos p\cos t
\Big)f(t)dt ,\quad f\in L^{2,e}(\mathbb{T},d\nu).
$$
%%%%%%%%%%%%%%%%%%%%%%%%%%%%%%%%%%%%%%%%%%%%%%%%%%%%%%%%%%%%%%%%%%%%%%%%%%%%%%%%%%%%%%

\section{Spectral properties of the operators $H_{\mu0},\,\mu \in \R$ and $H_{0\lambda},\lambda \in \R$}

%%%%%%%%%%%%%%%%%%%%%%%%%%%%%%%%%%%%%%%%%%%%%%%%%%%%%%%%%%%%%%%%%%%%%%%%%%%%%%%%%%%%%%%

Since the perturbation operator $V_{\mu0}$ resp.$V_{0\lambda}$ is of
rank $1$, according the well known Weyl's theorem the essential
spectrum $\sigma_{{ess}}(H_{\mu0})$
resp.$\sigma_{{ess}}(H_{0\lambda})$ of $H_{\mu0}$ resp.
$H_{0\lambda}$ doesn't depend on $\mu \in \R $ resp.$\lambda \in \R
$ and coincides to the spectrum ${\sigma}( H_{0})$ of $H_{0}$ (see
\cite{RSIV}), i.e.,
$$
 \sigma_{ess}(H_{\mu0})=\sigma_{ess}(H_{0\lambda})=\sigma(H_0)=
 [\min\limits_{p\in\mathbb{T}}\varepsilon(p),\,\max\limits_{p\in\mathbb{T}}\varepsilon(p)]= [0,2].
$$

For any $\mu,\lambda \in \R$ we introduce the Fredholm determinant
$\Delta(\mu,\lambda;z)$,  associating to the one particle
Hamiltonian $H_{\mu,\lambda}$, as follows
\begin{equation}\label{determinant}
\Delta(\mu,\lambda;z)=
\big[1-\mu a(z)\big]\big[1-\lambda c(z)\big]-\mu\lambda b^2(z),
\end{equation}
where
\begin{align}\label{defabc}
&a(z):=\int \limits_{\mathbb{T}}\frac{d\nu}{z-\varepsilon(q)},\\
&b(z):=-\int\limits_{\mathbb{T}}\frac{\cos qd\nu}{z-\varepsilon(q)},\\
&c(z):=\int\limits_ {\mathbb{T}}\frac{\cos^2qd\nu}{z-\varepsilon(q)}.
\end{align}
are regular functions in  $z \in \mathrm{C}\setminus [0,2]$.

In the following theorem we have collected  results on a unique
eigenvalue of the operators $H_{\mu0},\mu\in \R$ resp.
$H_{0\lambda},\lambda\in \R$ depending on the sign of  $\mu \neq 0$
resp. $\lambda\neq 0$.

\begin{theorem}\label{simple}
For any $0\neq\mu\in\R $ resp. $0 \neq\lambda\in\R $ the operator $H_{\mu0}$ resp. $H_{0\lambda}$ has a unique eigenvalue $\zeta(\mu)$ resp. $\zeta(\lambda)$ lying outside of the essential spectrum:
\begin{itemize}
\item[(\rm{i})]
If $\mu>0$ resp. $\lambda>0$, then the eigenvalue $\zeta(\mu)$ resp.
$\zeta(\lambda)$ lies in the interval $(2,+\infty).$
\item[(\rm{ii})] If $\mu<0$ resp.$\lambda<0$, then the eigenvalue
$\zeta(\mu)$ resp. $\zeta(\lambda)$ lies in the interval
$(-\infty,0).$
\item[(\rm{iii})] If $\mu>0$ resp.$\lambda<0$ then the eigenvalue $\zeta(\mu)$
resp. $\zeta(\lambda)$ lies in the interval $(2,+\infty)$ resp.
$(-\infty,0).$
\item[(\rm{iv})] If $\mu<0$ resp.$\lambda>0$ then the eigenvalue $\zeta(\mu)$
 resp. $\zeta(\lambda)$ lies in the interval $(-\infty,0)$ resp. $(2,+\infty).$
\end{itemize}
\end{theorem}

The proof of Theorem \ref{simple}  is a consequence of the formulated below Lemmas and
corollaries, which can be deduced from the simple properties of determinant
 $\Delta(\mu,0;z)$ resp.$\Delta(0,\mu;z)$.
\begin{lemma}\label{eig-zero}
The number $z\in \mathrm{C}\setminus [0,2]$ is an eigenvalue of the
operator $H_{\mu,0}$ resp. $H_{0,\lambda}$ if and only if
$\Delta(\mu,0;z)=0$ resp.\,$\Delta(0,\lambda;z)=0$.
\end{lemma}

\begin{lemma}\label{asyminfty}
Let $\mu,\lambda\in\R$.Then
\begin{align*}
&\lim\limits_{z\rightarrow \pm\infty}\Delta(\mu,0\,;z)=1,\\
&\lim\limits_{z\rightarrow \pm\infty}\Delta(0,\lambda\,;z)=1,\\
&\lim\limits_{z\rightarrow \pm\infty}\Delta(\mu,\lambda\,;z)=1.
\end{align*}
\end{lemma}
\begin{lemma}\label{asympt(abc)}
The functions $a(\cdot),\,\,\, b(\cdot),\,\,\,c(\cdot)$ are regular
in the region $\mathrm{C}\setminus [0,2]$,  positive and
monotonically decreasing in the intervals $(-\infty,0)$  and
$(2,+\infty)$ and the following asymptotics are true:
\begin{align*}
&a(z)=C_1(z-2)^{-\frac{1}{2}}+O(z-2)^\frac{1}{2}, as \,\,z\rightarrow 2+,\\
&b(z)=-C_1(z-2)^{-\frac{1}{2}}-1+O(z-2)^\frac{1}{2},as \,\,z\rightarrow 2+,\\
&c(z)=C_1(z-2)^{-\frac{1}{2}}-1+O(z-2)^\frac{1}{2},as \,\,z\rightarrow 2+,
\end{align*}
where $C_1>0$
and
\begin{align*}
&a(z)=-C_0(-z)^{-\frac{1}{2}}+O(-z)^{\frac{1}{2}},as\,\,z\rightarrow 0-,\\
&b(z)=-C_0(-z)^{-\frac{1}{2}}-1+O(-z)^{\frac{1}{2}},as\,\,z\rightarrow 0-,\\
&c(z)=-C_0(-z)^{-\frac{1}{2}}-1+O(-z)^{\frac{1}{2}},as\,\,z\rightarrow 0-,
\end{align*}
where $C_0>0.$
\end{lemma}
\begin{proof}
Since the functions under integral sign are positive  the monotonicity
of the Lebesgue integral gives  that the functions $a(z)$ and $c(z)$ are positive.
Now, we show that the function $$b(z):=-\int\limits_{\mathbb{T}}\frac{\cos qd\nu}{z-\varepsilon(q)}$$ is positive.
Representing $b(z)$ as
$$b(z)=-\int^{0}_{-\pi}\frac{\cos qd\nu}{z-\varepsilon(q)}-\int_{0}^{\pi}\frac{\cos qd\nu}{z-\varepsilon(q)}$$
and then changing of variables $q:= q+\pi$ we have that
$$b(z):=\int\limits^{\pi}_{0}\frac{2\cos^2qd\nu}{(z-1)^2-\cos^2q}>0$$

The asymptotics of functions $a(\cdot),\,\,
b(\cdot),\,\,c(\cdot)$ can be found in \cite{KhL2013}.
\end{proof}
The Lemma \ref{asympt(abc)} yields the following Corollary,
which gives asymptotics for the functions $\Delta(\mu,0;z)$ and $\Delta(0,\lambda;z)$.
\begin{corollary}\label{asymDelta(1+)}

The following asymptotics are true:
\begin{itemize}
\item[(\rm{i})] If\,\,$\mu,\,\lambda>0$. Then
\begin{align*}
&\lim\limits_{z\to 2+}\Delta(\mu,0;z)=-\infty,\\
&\lim\limits_{z\to 2+}\Delta(0,\lambda;z)=-\infty,\\
\end{align*}
\item[(\rm{ii})] If \,\, $\mu,\,\lambda<0$.Then
\begin{align*}
&\lim\limits_{z\to 2+}\Delta(\mu,0;z)=+\infty,\\
&\lim\limits_{z\to 2+}\Delta(0,\lambda;z)=+\infty,\\
\end{align*}
\item[(\rm{iii})] If \,\, $\mu,\lambda>0$. Then
\begin{align*}
&\lim\limits_{z\rightarrow 0-}\Delta(\mu,0;z)=+\infty,\\
&\lim\limits_{z\rightarrow 0-}\Delta(0,\lambda;z)=+\infty,\\
\end{align*}
\item[(\rm{iv})] If \,\, $\mu,\lambda<0$. Then
\begin{align*}
&\lim\limits_{z\rightarrow 0-}\Delta(\mu,0;z)=-\infty,\\
&\lim\limits_{z\rightarrow 0-}\Delta(0,\lambda;z)=-\infty,\\
\end{align*}
\end{itemize}
\end{corollary}

\section{Spectral properties of the operators $H_{\mu\lambda},\,\mu,\lambda \in \R$.}

The perturbation operator $V_{\mu\lambda},\,\mu,\lambda \in \R$ is
of rank $2$ and hence by the well known Weyl's theorem the essential
spectrum $\sigma_{{ess}}(H_{\mu\lambda})$ of $H_{\mu\lambda}$
doesn't depend on $\mu,\lambda \in \R $ and coincides to the
spectrum ${\sigma}( H_{0})$ of $H_{0}$ (see \cite{RSIV}), i.e.,
$$
 \sigma_{ess}(H_{\mu\lambda})=\sigma(H_0)=
 [\min\limits_{p\in\mathbb{T}}\varepsilon(p),\,\max\limits_{p\in\mathbb{T}}\varepsilon(p)]= [0,2].
$$

\begin{remark}
Note that since
$$
(V_{\mu\lambda}f,f)=\mu |\int\limits_{\mathbb{T}}f(t)d\nu|^2
+\lambda |\int\limits_{\mathbb{T}}\cos tf(t)d\nu|^2, \quad f\in
L^{2,e}(\mathbb{T},d\nu),
$$
the operator $V_{\mu\lambda}$ is not only positive or only negative  and hence
the operator  $H_{\mu\lambda}$ may have eigenvalues as below
the bottom of the essential spectrum, as well as above the its top. \end{remark}
The following lemma describes the relations between the operator $H_{\mu,\lambda}$ and determinant
$\Delta(\mu,\lambda;z)$ defined in \eqref{determinant}.
\begin{lemma}\label{eig-zero}
The number $z\in \mathrm{C}\setminus [0,2]$ is an eigenvalue
of the operator $H_{\mu,\lambda}$ if and only if $\Delta(\mu,\lambda;z)=0$.
\end{lemma}

\begin{proof}
Let the operator $H_{\mu,\lambda}$ has an eigenvalue $z\in \mathrm{C}\setminus [0,2]$, i.e.,
the equation
\begin{equation}\label{eýg-equation}
(z-H_{\mu,\lambda})\psi(q)=(z-\varepsilon(q))\psi(q)- \mu\int\limits_{\T} \psi(t)
  d\nu(t)-\lambda \cos p\int\limits_{\T} \cos t \psi(t)d\nu(t)=0
\end{equation}
has a non-zero solution $\psi \in L^{2,e}(\T,d\nu)$. We introduce
the following linear continuous functionals defined on the Hilbert
space $\psi \in L^{2,e}(\T,d\nu)$
\begin{align}\label{denotation}
c_1:=c_1(\psi):=\int\limits_{\T} \psi(t)d\nu(t)\\
c_2:=c_2(\psi):=\int\limits_{\T}\cos t \psi(t).
\end{align}
Then we easily find that the solution of the equation \eqref{eýg-equation} has form
\begin{equation}\label{express}
\psi(q)= \mu \frac{c_1}{z-\varepsilon(q)}+\lambda\frac{c_2\cos q}{z-\varepsilon(q)}.
\end{equation}
Putting the expression \eqref{express} for $\psi$ to
\eqref{denotation} and (4.7) we get the following homogeneous system
of linear equations with respect to the functionals $c_1$ and $c_2$
\begin{equation}\label{system}
\left\lbrace\begin{array}{ccc}
c_1= \mu c_1\int\limits_{\T} \dfrac{d\nu}{z-\varepsilon(q)}+\lambda c_2\int\limits_{\T}\dfrac{\cos qd\nu}{z-\varepsilon(q)}\\
c_2= \mu c_1\int\limits_{\T} \dfrac{\cos q
d\nu}{z-\varepsilon(q)}+\lambda c_2
\int\limits_{\T}\dfrac{\cos^2qd\nu}{z-\varepsilon
(q)}\\
\end{array}\right.
\end{equation}
Hence, we can conclude that this homogenous system of linear
equations has nontrivial solutions if and only if  the associated
determinant $\Delta(\mu,\lambda;z)$ has zero $z\in
\mathbf{C}\setminus [0,2]$.

On the contrary, let a number $z\in \mathbf{C}\setminus [0,2]$ be a
zero of determinant $\Delta(\mu,\lambda;z)$. Then it easily can be
checked that $z$ is eigenvalue of $H_{\mu,\lambda}$ and the function
\begin{equation}\label{expression} \psi(q)= \mu
\dfrac{c_1}{z-\varepsilon(q)}+\lambda\dfrac{c_2\cos
q}{z-\varepsilon(q)},
\end{equation}
is the associated eigenfunction, where the  vector $(c_1,c_2)$ is a
non-zero solution of the system \eqref{system}.
\end{proof}
The  following asymptotics for the determinant
$\Delta(\mu,\lambda,z)$ can be received applying the asymptotics of
the functions $a(\cdot),\,\,\, b(\cdot),\,\,\,c(\cdot)$ in Lemma
\ref{asympt(abc)}.
\begin{lemma}\label{asympDELTA}
\begin{align}
&\Delta(\mu,\lambda,z)=C^{+}_{-\frac{1}{2}}(\mu,\lambda)
(z-2)^{-\frac{1}{2}}+C^{+}_0(\mu,\lambda)+O(z-2)^{\frac{1}{2}},as \,\, z\rightarrow 2+,\\
&\Delta(\mu,\lambda,z)=C^{-}_{-\frac{1}{2}}(\mu,\lambda)
(-z)^{-\frac{1}{2}}+C^{-}_0(\mu,\lambda)+O(-z)^\frac{1}{2}, as \,\,z\rightarrow 0-,
\end{align}
where
\begin{align}
&C^{+}_{-\frac{1}{2}}(\mu,\lambda)= B_2(\mu\lambda-\mu-\lambda),\,B_2>0\\
&C^{+}_0(\mu,\lambda)=1+\lambda-\mu\lambda,\\
&C^{-}_{-\frac{1}{2}}(\mu,\lambda)= B_0(\mu\lambda+\mu+\lambda),\,B_0>0\\
&C^{-}_0(\mu,\lambda)=1-\lambda-\mu\lambda.
\end{align}

\end{lemma}
The Lemma \ref{asympDELTA} yields the following results for the
determinant $\Delta(\mu,\lambda;z)$.

\begin{corollary}\label{asymDELTA(1)}
For the determinant $\Delta(\mu,\lambda;z)$ the following results
are true:
\begin{itemize}
\item[(\rm{i})]
Assume $C^{+}_{-\frac{1}{2}}(\mu,\lambda)>0$ and
$C^{-}_{-\frac{1}{2}}(\mu,\lambda)>0$. Then
\begin{align*}
&\lim\limits_{z\rightarrow 2+}\Delta(\mu,\lambda;z)=+\infty.\\
&\lim\limits_{z\rightarrow 0-}\Delta(\mu,\lambda;z)=+\infty.\\
\end{align*}
\item[(\rm{ii})]
Assume $C^{+}_{-\frac{1}{2}}(\mu,\lambda)=0,\,\,\mu>1$ and
$C^{-}_{-\frac{1}{2}}(\mu,\lambda)=0,\,\,\mu<-1$. Then
\begin{align*}
&\lim\limits_{z\rightarrow 2+}\Delta(\mu,\lambda;z)<0,\\
&\lim\limits_{z\rightarrow 0-}\Delta(\mu,\lambda;z)<0.\\
\end{align*}
\item[(\rm{iii})]
Assume $C^{+}_{-\frac{1}{2}}(\mu,\lambda)<0$ and
$C^{-}_{-\frac{1}{2}}(\mu,\lambda)<0$.Then
\begin{align*}
&\lim\limits_{z\rightarrow 2+}\Delta(\mu,\lambda;z)=-\infty,\\
&\lim\limits_{z\rightarrow 0-}\Delta(\mu,\lambda;z)=-\infty.\\
\end{align*}
\item[(\rm{iv})]
Assume $C^{+}_{-\frac{1}{2}}(\mu,\lambda)=0,\,\,\mu <1$ and
$C^{-}_{-\frac{1}{2}}(\mu,\lambda)=0,\,\,\mu>-1$. Then
\begin{align*}
&\lim\limits_{z\rightarrow 2+}\Delta(\mu,\lambda;z)>0,\\
&\lim\limits_{z\rightarrow 0-}\Delta(\mu,\lambda;z)>0.\\
\end{align*}
\end{itemize}
\end{corollary}

To formulate the main theorem we introduce the regions
$\mathbb{G}_{2,+},\mathbb{G}_{1,+}$ and $\mathbb{G}_{0,+}$
associated to the function $C^{+}_{-\frac{1}{2}}(\mu,\lambda)$  and
also the regions $\mathbb{G}_{2,-},\mathbb{G}_{1,-}$ and
$\mathbb{G}_{0,-}$ associated to the function
$C^{-}_{-\frac{1}{2}}(\mu,\lambda)$ as follows
\begin{align*}
&\mathbb{G}_{2,+}=\{(\mu,\lambda)\in\R^2:\,C^{+}_{-\frac{1}{2}}(\mu,\lambda)>0,\,\,\mu>1\},\\
&\mathbb{G}_{1,+}=\{(\mu,\lambda)\in\R^2:\,C^{+}_{-\frac{1}{2}}(\mu,\lambda)=0,\,\,\mu>1
\,\,\mbox{or}\,\,\,C^{+}_{-\frac{1}{2}}(\mu,\lambda)< 0\},\\
&\mathbb{G}_{0,+}=\{(\mu,\lambda)\in\R^2\}:C^{+}_{-\frac{1}{2}}(\mu,\lambda)=0,\,\,\mu<1\,\,\mbox{or}\,\,\,
C^{+}_{-\frac{1}{2}}(\mu,\lambda)>0 \\
\end{align*}
and
\begin{align*}
&\mathbb{G}_{2,-}=\{(\mu,\lambda)\in\R^2:\,C^{-}_{-\frac{1}{2}}(\mu,\lambda)>0, \,\,\mu<-1,\}\\
&\mathbb{G}_{1,-}=\{(\mu,\lambda)\in\R^2:\,C^{-}_{-\frac{1}{2}}(\mu,\lambda)=0,
\,\mu<-1\,\,\mbox{or}\,\,\,C^{-}_{-\frac{1}{2}}(\mu,\lambda)<0\},\\
&\mathbb{G}_{0,-}=\{(\mu,\lambda)\in\R^2\}:C^{+}_{-\frac{1}{2}}(\mu,\lambda)=0,\,\,\mu>-1
\,\,\mbox{or}\,\,\,C^{-}_{-\frac{1}{2}}(\mu,\lambda)>0.\\
\end{align*}

The main results are given in the following theorem, where the
existence and location of eigenvalues of the one-particle
Hamiltonian $H$  with indefinite sign interaction $v_{\mu\lambda}$
are stated.

The Hamiltonian $H_{\mu\lambda}$ can have one or two eigenvalues,
situating as below the bottom of the essential spectrum, as  well as
above its top. Moreover, the operator $H_{\mu\lambda}$ has two
eigenvalues outside of the essential spectrum, depending on
$\mu\neq0$ and $\lambda\neq0$, where one of them is situated below
the bottom of the essential spectrum and the other one  above  its
top.
 \begin{figure}[figure.bmp]

\mbox{\includegraphics[bb= 0 0 9.85cm 9.95cm]{figure.bmp}}\\

 \caption{}
 \end{figure}

\begin{theorem}\label{main}

\begin{itemize}
\item[(\rm{i})]
Assume $(\mu,\lambda)\in
\mathrm{G}_{02}=\mathbb{G}_{0,-}\cap\mathbb{G}_{2,+}$. Then the
operator $H_{\mu\lambda}$ has no eigenvalue below the essential
spectrum and it has two eigenvalues $\zeta_1(\mu,\lambda)$ and
$\zeta_2(\mu,\lambda)$ satisfying the following relations
\begin{equation*}\label{main1}
2<\zeta_1(\mu,\lambda)<\zeta_{\min}(\mu,\lambda)\leq\zeta_{\max}(\mu,\lambda)<\zeta_2(\mu,\lambda).
\end{equation*}
\item[(\rm{ii})]
Assume $(\mu,\lambda)\in
\mathrm{G}_{01}=\mathbb{G}_{0,-}\cap\mathbb{G}_{1,+}$. Then the
operator $H_{\mu\lambda}$ has no eigenvalue below the essential
spectrum and it has one eigenvalue  $\zeta_2(\mu,\lambda)$
satisfying the following relation\\ $\zeta_2(\mu,\lambda)>2$.

\item[(\rm{iii})] Let $(\mu,\lambda)\in \mathrm{G}_{11}=\mathbb{G}_{1,-}\cap\mathbb{G}_{1,+}$.
Then the operator $H_{\mu\lambda}$ has two eigenvalues
$\zeta_1(\mu,\lambda)$ and $\zeta_2(\mu,\lambda)$ satisfying the
following relations  $$\zeta_1(\mu,\lambda)<0\,\, \mbox{and} \,\,
\zeta_2(\mu,\lambda)>2.$$

\item[(\rm{iv})] Assume $(\mu,\lambda)\in \mathrm{G}_{10}= \mathbb{G}_{1,-}
\cap\mathbb{G}_{0,+}.$  Then the operator $H_{\mu\lambda}$ has one
eigenvalue  $\zeta_1(\mu,\lambda)$ satisfying the relation
$\zeta_1(\mu,\lambda)<0$  it has no eigenvalue above  the essential
spectrum.

\item[(\rm{v})] Assume $(\mu,\lambda)\in \mathrm{G}_{20}=\mathbb{G}_{2,-}\cap\mathbb{G}_{0,+}$.
Then the operator $H_{\mu\lambda}$   has two  eigenvalues
$\zeta_1(\mu,\lambda)$ and $\zeta_2(\mu,\lambda)$ satisfying the
following relations
$$
\zeta_1(\mu,\lambda)<\zeta_{\min}(\mu,\lambda)\leq\zeta_{\max}(\mu,\lambda)<\zeta_2(\mu,\lambda)<0
$$ and it has no eigenvalue above the essential spectrum.
\end{itemize}
\end{theorem}
\begin{remark}
The sets $\mathrm{G}_{02}, \mathrm{G}_{01}, \mathrm{G}_{11},
\mathrm{G}_{10}\,\, \mbox{and}\,\, \mathrm{G}_{20}$ which appears in
Theorem \ref{main} are shown in the figure $1$.
\end{remark}

\begin{remark}
We remark that if $\varepsilon(\cdot)$ is arbitrary real-analytic
function on $\T^d$ and has a unique non-degenerate minimum and
maximum, then Theorem \ref{main} holds.
\end{remark}

\begin{proof}
\begin{itemize}
\item[(\rm{i})]
Assume $(\mu,\lambda)\in (\mu,\lambda)\in \mathrm{G}_{02}$ and
$z<0$. Then an application the Cauchy--Schwarz inequality for the
functions $[\varepsilon(q)-z]^{-\frac{1}{2}}$ and $\cos q\,
[\varepsilon(q)-z]^{-\frac{1}{2}}$ yields the inequality
\begin{align*}\label{determinant}
&\Delta(\mu,\lambda;z)=\big(1+\mu\int \limits_{\mathbb{T}}\frac{d\nu}{\varepsilon
(q)-z}\big)+\big(1+\lambda\int\limits_ {\mathbb{T}}\frac{\cos^2q d\nu}{\varepsilon
(q)-z}\big)\\
&+\mu\lambda\big[\int \limits_{\mathbb{T}}\frac{d\nu}{\varepsilon
(q)-z}\int \limits_{\mathbb{T}}\frac{\cos^2qd\nu}{\varepsilon
(q)-z}-\big(\int\limits_{\mathbb{T}}\frac{\cos q d\nu}{\varepsilon(q) -z}\big)^2\big]>0,
\end{align*}
i.e., $\Delta(\mu,\lambda;z)$ has  no zero in the interval $(-\infty,0)$.
 Lemma \ref{eig-zero} gives that the operator $H_{\mu\lambda}$ has no
 eigenvalue below the bottom of the essential spectrum.

Let $(\mu,\lambda)\in \mathrm{G}_{02}$ and $z>2$.

Since $\mu,\lambda>0$ the function
$\Delta(\mu,0;\cdot)$\quad\mbox{resp.}\quad
$\Delta(0,\lambda;\cdot)$ is monotone increasing in $(1,+\infty).$
Applying Lemma \ref{asyminfty} we have
$$\lim\limits_{z\to +\infty}\Delta(\mu,0;z)=1
\,\,\mbox{resp.}\,\,\lim\limits_{z\to +\infty}\Delta(0,\lambda;z)=1.$$
Corollary \ref{asymDelta(1+)} gives that $$\lim\limits_{z\rightarrow 1+}\Delta(\mu,0;z)=-\infty,\,\,\mbox{resp.}\,\,\
\lim\limits_{z\rightarrow 1+}\Delta(0,\lambda;z)=-\infty.$$

The continuous function $\Delta(\mu,0;\cdot)$\quad\mbox{and}\quad
$\Delta(0,\lambda;\cdot)$  has a zero $\zeta(\mu)$\, resp.
\,$\zeta(\lambda)$ in the interval $(1,+\infty)$. The representation
\eqref{determinant} of the determinant $\Delta(\mu,\lambda\,;z)$
gives the inequality $\Delta(\mu,\lambda\,;\zeta(\mu))<0$ resp.
$\Delta(\mu,\lambda\,;\zeta(\lambda))<0$. Denote by
\begin{align*}
&\zeta_{\min}(\mu,\lambda)=\min\{\zeta(\mu),\zeta(\lambda)\}\\
&\zeta_{\max}(\mu,\lambda)=\max\{\zeta(\mu),\zeta(\lambda)\}.
\end{align*}

The representation \eqref{determinant} of  determinant $\Delta(\mu,\lambda;z)$ gives the inequality
$\Delta(\mu,\lambda;\zeta_{\min}(\mu,\lambda))<0$.
Corollary \ref{asympDELTA} yields
$$\lim\limits_{z\to 1+}\Delta(\mu,\lambda;z)=+\infty$$
Hence  there exist a number
$z_1(\mu,\lambda)\in (1,\zeta_{\min}(\mu,\lambda))$ such that
\begin{equation*}
\Delta(\mu,\lambda;z_1(\mu,\lambda;0))=0.
\end{equation*}
Lemma \ref{eig-zero} gives the existence of the eigenvalue of the operator
in the interval $(1,\zeta_{\min}(\mu,\lambda))$.

The monotonicity of function $\Delta(\mu,0;z)$ resp. $\Delta(\lambda,0;z)$
 gives for $z>\zeta(\mu)$ resp. $z>\zeta(\lambda)$  the relation
$$\Delta(\mu,0;z)>\Delta(\mu,0;\zeta(\mu))=0,\,\, \mbox{resp.}\,\,
\Delta(\lambda,0;z)>\Delta(\lambda;\zeta(\lambda))=0.
$$
Applying Lemma \ref{asympt(abc)} we have in the interval $(2,+\infty)$ the inequality
\begin{equation*}
\frac{\partial \Delta(\mu,\lambda;z)}{\partial z}=-\mu
\Delta(0,\lambda;z)a'(z)-\lambda \Delta(\mu,0,;z) c'(z)- 4\mu\lambda
b(z)b'(z)>0,
\end{equation*}
i.e., the function $\Delta(\mu,\lambda;\cdot)$ is monotone
increasing in the interval $(\zeta_{\max}(\mu,\lambda),+\infty)$.
Lemma \ref{asyminfty}, i.e., the relation
\begin{equation*}
\lim\limits_{z\to +\infty} \Delta(\mu,\lambda;z)=1,
\end{equation*}
yields  the existence a unique number $z_2(\mu,\lambda)\in
(\zeta_{\max}(\mu,\lambda),+\infty $ such that
\begin{equation*}
\Delta(\mu,\lambda;z_2(\mu,\lambda;0))=0.
\end{equation*}
Lemma \ref{eig-zero} gives that the operator has two eigenvalues above the top of the essential spectrum.
These eigenvalues obeys the relations \eqref{main1}.

\item[(\rm{ii})] Assume $(\mu,\lambda)\in
\mathrm{G}_{01}$ and $z<0$.

As in the case $(\rm{i})$ we can show that operator $H_{\mu\lambda}$
has no eigenvalue below the essential spectrum.

It is easy to  show that for any $\mu>0$ the operator $H_{\mu0}$ has
only one eigenvalue at the point $(\mu,0)\in \mathrm{G}_{02}$.

Indeed. Lemma \ref{asyminfty} and Corollary \ref{asymDELTA(1)} give
that
$$\lim\limits_{z\to +\infty }\Delta(\mu,0;z)=1$$ and
$$\lim\limits_{z\to 2+}\Delta(\mu,0;z)<0.$$
Hence, the continuous function $\Delta(\mu,0;\cdot)$ in $z\in
(2,+\infty)$ has a unique zero $\zeta_1(\mu,0)\in (2,+\infty).$

Let $(\mu,\lambda)\in \mathrm{G}_{01}$ be an other point, then there
is a line(curve) $\Gamma [(\mu,0),(\mu,\lambda)]\in
\mathrm{G}_{01}$, which connects the points $(\mu,0)$ and
$(\mu,\lambda)$. The compactness of $\Gamma
[(\mu,0),(\mu,\lambda)]\in \mathrm{G}_{01}$ yields that  at the
point $(\mu,\lambda)\in \mathrm{G}_{01}$ the function
$\Delta(\mu,\lambda;z)$ has only one zero. Thus, Lemma
\ref{eig-zero} yields that the operator has only one eigenvalue
above the top of the essential spectrum.

\item[(\rm{iii})] Assume $(\mu,\lambda)\in \mathrm{G}_{11}$.

In this case applying Lemma \ref{asyminfty} and Corollary
\ref{asymDELTA(1)} give
$$\lim\limits_{z\to 2+}\Delta(\mu,\lambda;z)=-\infty,$$
$$\lim\limits_{z\to 0-}\Delta(\mu,\lambda;z)=-\infty,$$
and
$$\lim\limits_{z\to \pm\infty}\Delta(\mu,\lambda;z)=1.$$
Hence, the continuous function $\Delta(\mu,\lambda;\cdot)$ in $z\in
(-\infty,0)\cup (2,+\infty)$ has two zeros $\zeta_1(\mu,\lambda)\in
(-\infty,0)$ and $\zeta_2(\mu,\lambda)\in (2,+\infty)$.

Thus, Lemma \ref{eig-zero} yields that the operator has two
eigenvalues: one of them lies below the bottom of the essential
spectrum and other one lies above its top.

The other cases $(iv)$ and $(v)$ of Theorem \ref{main}  can be proven by
the same way as the cases  $(i)$ and $(ii)$.
\end{itemize}
\end{proof}
\section{Acknowledgment}
The work was supported by the Fundamental Science Foundation of
Uzbekistan.

\end{document}